\documentclass[12pt]{amsart}
\usepackage[english]{babel}
\usepackage{amsfonts,amssymb,latexsym,amscd}

\theoremstyle{plain}
\newtheorem{theorem}{Theorem}

\newtheorem{proposition}{Proposition}
\newtheorem{corollary}{Corollary}
\theoremstyle{definition}
\newtheorem{definition}{Definition}
\theoremstyle{remark}
\newtheorem{remark}{Remark}
\newtheorem{example}{Example}

\usepackage{hyperref}

\begin{document}

\title[Uniformly Movable Categories]{Uniformly Movable Categories And Uniform Movability Of Topological Spaces}

\author{Pavel S. Gevorgyan}
\address{Moscow Pedagogical State University}
\email{pgev@yandex.ru}

\author{I. Pop}
\address{Faculty of Mathematics, ''Al. I. Cuza'' University, 700505-Ia\c{s}i, Romania}
\email{ioanpop@uaic.ro}

\begin{abstract}
A categorical generalization of the notion of movability from the inverse systems and shape theory was given by the first author who defined the notion of movable category and interpreted by this the movability of topological spaces. In this paper the authors define the notion of uniformly movable category and prove that a topological space is uniformly movable in the sense of the shape theory if and only if its comma category in the homotopy category HTop over the subcategory HPol of polyhedra is a uniformly movable category. This is a weakened version of the categorical notion of uniform movability introduced by the second author.
\end{abstract}



\maketitle

\section{Introduction}

The notion of movability for metric compacta was introduced by K.
Borsuk \cite{B} as an important shape invariant. The movable spaces are a
generalization of spaces having the shape of ANR's. The property of
movability allows that a series of important results in algebraic
topology (as Whitehead and Hurewicz theorems) remain valid with the
homotopy pro-groups replaced by the corresponding shape groups. 

The term movability comes from a geometric interpretation of the definition in the compact case: if $X$ is a compactum lying in a space $M\in $ AR, one says that $X$ is movable if for every neighborhood $U$ of $X$ in $M$ there exists a neighborhood $V$ of $X$, $V\subset U$, such that for every neighborhood $W$ of $X$, $W\subset U$, there is a homotopy $H :V\times [0,1]\rightarrow U$ such that $H(x,0)=x$ and $H(x,1)\in W$, for every $x\in V$. One shows that the choice of the space $M\in $AR does not matter \cite{B}. After the notion of movability had been expressed in terms of ANR-systems, for arbitrary topological spaces, \cite{MS1}, it became clear that one could define it in arbitrary pro-categories. 

The definition of a movable object in an arbitrary pro-category and the definition of the uniform movability was given by M. Moszy\'{n}ska \cite{MOS}. The uniform movability is important in the study of mono- and epi- morphisms in pro-categories and in the study of the shape of pointed spaces. In the book of Marde\v{s}i\'{c} \& Segal \cite{MS2} all these approaches and applications of various types of movability are discussed.

A categorical generalization of the notion of movability from inverse systems and the shape theory was given by the first author of the present paper who defined the notion of movable category and interpreted by this property the movability of topological spaces \cite{G1}.

Regarding a concept of uniform movability for a category, an approach was
given by the second author in \cite{P}. In that paper a category $\mathcal{K}$ is
called uniformly movable with respect to a subcategory $\mathcal{K}^{\prime }
$ provided that there exists a pair $(F,\varphi )$ with $F:\mathcal{%
K\rightarrow K}$ a covariant functor and $\varphi :$ $F\rightarrow 1_{%
\mathcal{K}}$ a natural transformation such that every morphism $f\in
\mathcal{K(}Y,X)$ in the category $\mathcal{K},$ with $Y\in \mathcal{K}%
^{\prime },$ admits a morphism $G(f)\in \mathcal{K(}F(X),Y)$
satisfying the
relation $f\circ G(f)=\varphi (X)$ and such that the correspondence $%
f\rightarrow G(f)$ is natural in the sense that a commutative diagram
\begin{center}
$
\begin{array}{lll}
Y & \stackrel{f}{\rightarrow } & X \\
\downarrow _{u} &  & \downarrow _{v} \\
Y^{\prime } & \stackrel{f^{\prime }}{\rightarrow } & X^{\prime }
\end{array}
$
\end{center}
in the category $\mathcal{K}$, with $u:Y\rightarrow Y^{\prime }$ a morphism
in $\mathcal{K}^{\prime }$, induces the equality $G(f^{\prime })\circ
F(v)=u\circ G(f).$ In the case $\mathcal{K}^{\prime }=\mathcal{K}$ the
category $\mathcal{K}$ was simply named uniformly movable category. The pair
$(F,\varphi )$ is called a uniform movability pair of $\mathcal{K}$ and the
morphism $G(f)$ a uniform movability factor of $f.$

This definition is good and suggestive from the functorial point of
view but it is too strong if one has in mind movability of
topological spaces. The uniform movability of the comma category of
a space $X$ in HTop over HPol implies the uniform movability of $X$
(\cite{P}, Cor. 1). However, the converse was proved only with two
supplementary conditions for the space $X$ (\cite{P}, Cor. 2). In the
present paper we give an weakened version of this definition so that
the new definition permits a characterization of the uniform
movability of an arbitrary space $X $ via the uniform movability of
the comma category of $X.$ This is in fact the subject of the
present paper.

\section{Uniformly movable categories}

Let $\mathcal{K}$ be an arbitrary category.

\begin{definition}[{\cite{G1},\cite{G2}}] We say that an object $X$ of the category $\mathcal{K}$ is \emph{movable},
if there are an object $M(X)\in \mathcal{K}$ and a morphism $m_{X}:M(X)\rightarrow X$ in $\mathcal{K}$ that satisfy the condition: for any object $Y\in \mathcal{K}$ and any morphism $p:Y\rightarrow X$ in $%
\mathcal{K}$ there exists a morphism $u(p):M(X)\rightarrow Y$
which makes the diagram

\begin{center}
\begin{picture}(70,90)
\put(0,35){$X$} \put(70,0){$Y$} \put(60,70){$M(X)$}
\put(63,65){\vector(-2,-1){48}} \put(26,56){\scriptsize $m_X$}
\put(65,7){\vector(-2,1){51}} \put(35,12){\scriptsize $p$}
\put(74,66){\vector(0,-1){53}} \put(76,35){\scriptsize $u(p)$}
\end{picture}

Diagram 1.
\end{center}
commutative i. e., $p\circ u(p)=m_{X}.$

A category $\mathcal{K}$ is called \emph{movable} if any object of the category $%
\mathcal{K}$ is movable.
\end{definition}

\begin{definition}\label{umov}
We say that an object $X$ of the category $\mathcal{K}$ is \emph{uniformly movable}, if there are an object $M(X)\in \mathcal{K}$ and a morphism $m_{X}:M(X)\rightarrow X$ in $\mathcal{K}$ that satisfy the following conditions:

1. for any object $Y\in \mathcal{K}$ and any morphism $p:Y\rightarrow X$ in $%
\mathcal{K}$ there exists a morphism $u(p):M(X)\rightarrow Y$ in $\mathcal{K}
$ which make Diagram 1 commutative and

2. for all morphisms $p:Y\rightarrow X,q:Z\rightarrow X$ and $r:Z\rightarrow
Y$ in $\mathcal{K}$, such that $p\circ r=q,$ the Diagram 2 is commutative

\begin{center}
\begin{picture}(70,90)
\put(0,0){$Z$} \put(0,70){$Y$} \put(20,35){$X$} \put(70,35){$M(X)$}

\put(3,15){\vector(0,1){48}} \put(-3,35){\scriptsize $r$}
\put(10,70){\vector(1,-2){11}} \put(19,58){\scriptsize $p$}
\put(10,8){\vector(1,2){11}} \put(17,15){\scriptsize $q$}
\put(67,40){\vector(-1,0){35}} \put(42,34){\scriptsize $m_X$}
\put(68,47){\vector(-2,1){55}} \put(40,62){\scriptsize $u(p)$}
\put(68,30){\vector(-2,-1){55}} \put(40,10){\scriptsize $u(q)$}
\end{picture}

Diagram 2.
\end{center}
i. e., $u(p)=r\circ u(q).$

A category $\mathcal{K}$ is called \emph{uniformly movable} if any object of the category $\mathcal{K}$ is uniformly movable.

We will call the morphism $m_{X}$ a \emph{(uniform) movability morphism} of $X$ and the morphism $u(p)$ a \emph{(uniform) movability factor} of $p.$
\end{definition}

It is evident that every uniformly movable object is movable. The following example shows that the converse is not true.

\begin{example}
Let \textbf{Set}$^{\circ }$ be the category of nonempty sets. Then a
singleton is a movable object but not uniformly movable.
\end{example}

Indeed, let $\{*\}$ ba a singleton. Then it is obvious that we can take  as a movability morphism for $\{*\}$ any constant map $M(*)\to \{*\}$. For an arbitrary map $q:Z\to \{*\}$ we can take $u(q):M(*)\to Z$ any map. Thus $\{*\}$ is movable. But if $p:Y\to \{*\}$ is another map we can write $q=p\circ r$, $q=p\circ r'$, for any maps $r,r':Z\to Y$. Let $x_*\in M(*)$ and suppose that $r(u(q)(x_*))\neq r'(u(q)(x_*))$ (suppose that $Y$ has the cardinal at least two). Then the relations $u(p)=r\circ u(q)$ and $u(p)=r'\circ u(q)$ are incompatible. Therefor $\{*\}$ is not uniformly movable.

\begin{remark}
If $\mathcal{K}$ is a uniformly movable category in the sense of \cite{P} (see
also Introduction) then $\mathcal{K}$ is also a uniformly movable category
in the sense of Definition 2.
\end{remark}
Indeed, let $(F,\varphi )$ be a uniform movability pair
of the category $\mathcal{K}$ and for a morphism $f\in \mathcal{K}(X,Y),$ $G(f)$ a uniform movability factor of $f$ (see Introduction or Def. 1 in \cite{P}). Now with the notations from Definition 2, if $X$ is an object in $%
\mathcal{K}$, and $p:Y\rightarrow X$ is a morphism in $\mathcal{K}$, we
take $M(X)=F(X),m_{X}=\varphi (X)$ and $u(p)=G(p).$ Now the relation $p\circ
G(p)=\varphi (X)$ is translated as $p\circ u(p)=m_{X}$ which is the
condition 1. from Definition 2. Then for some morphisms $p:Y\rightarrow
X,q:Z\rightarrow X$ and $r:Z\rightarrow Y$ such that $p\circ r=q$, we can
consider the commutative diagram
$$
\begin{array}{lll}
Y & \stackrel{p}{\rightarrow } & X \\
\uparrow _r &  & \uparrow _{1_{X}} \\
Z & \stackrel{q}{\rightarrow } & X
\end{array}
$$
such that we have the equality $r\circ G(p)=G(q)$ which is translated as $%
r\circ u(p)=u(q),$ which is the condition 2 from Definition
\ref{umov}.

This remark permits to take over a series of examples from \cite{P}.

\begin{proposition}
Every category $\mathcal{K}$ with null morphisms is a uniformly
movable category.
\end{proposition}

\begin{proof}
For each object $X\in \mathcal{K}$ we can put $M(X)=X_{0},$
for a fixed object $X_{0}$, with $m_{X}=0_{X_{0}X}:X_{0}\rightarrow X,$ the
null morphism from $X_{0}$ to $X,$ and for an arbitrary morphism $%
p:Y\rightarrow X,$ $u(p)=0_{X_{0}Y}$. Now it is not difficult to
verify Conditions 1 and 2 of Definition \ref{umov}.

Particularly, the category \textbf{Set}$_{*}$ of pointed sets is uniformly
movable.
\end{proof}

\begin{proposition}\label{pr3}
Every category $\mathcal{K}$ with an initial object $O$ is a
uniformly movable category.
\end{proposition}

\begin{proof}
For each object $X\in \mathcal{K}$, we can put $M(X)=O$, with $m_{X}:O\rightarrow X$, the only element in the set
$\mathcal{K}(O,X),$ and for an arbitrary morphism $p:Y\rightarrow
X$, $u(p):O\rightarrow Y$, the only element in the set
$\mathcal{K}(O,Y)$. Now it is not difficult to verify Conditions 1
and 2 of Definition \ref{umov}.
\end{proof}

Particularly, the categories \textbf{Set} of all sets and maps and
\textbf{\ Gr} of all groups and homomorphisms are uniformly movable
categories.

\begin{proposition}
An object dominated by a uniformly movable object is uniformly movable.
\end{proposition}

\begin{proof}
Let $X $ be a uniformly movable object in a category $\mathcal{K}$ and $Y$ an object dominated by $X$: $Y\leq X$. Let us consider the morphisms $f:X\rightarrow Y$ and $g:Y\rightarrow X,$ with $f\circ g=1_{Y}$. Now we can take $M(Y)=M(X)$ and $m_{Y}=f\circ m_{X}$. For arbitrary morphism $p:Z\rightarrow Y$ let us define $u(p)=u(g\circ p)$. Then by the relation $(g\circ p)\circ u(g\circ p)$ $=m_{X}$ it follows ($f\circ g)\circ p\circ u(g\circ p)=f\circ m_{X}$ and thus $p\circ u(p)=m_{Y},$ which is Condition 1 of Definition \ref{umov}. Now let the morphisms $p:Z\rightarrow Y$, $q:U\rightarrow Y$, $r:U\rightarrow Z$ satisfy the relation $q=p\circ r$. Then we have $g\circ q=(g\circ p)\circ r,$ which implies $u(g\circ p)=r\circ u(g\circ q)$, i.\ e., Condition 2 of Definition \ref{umov} holds.
\end{proof}

\begin{definition}[{\cite{G2}}]
We say that a category $\mathcal{L}$ is \emph{weakly functorially dominated} by a category $\mathcal{K}$ if there are functors
$J:\mathcal{L}\to \mathcal{K}$ and $D:\mathcal{K}\to \mathcal{L}$
and a natural transformation $\psi :D\circ J\rightarrow
1_{\mathcal{L}}$.
\end{definition}

The following proposition is similar to Theorem 2 from \cite{G2}.

\begin{proposition}\label{pr1}
If the category $\mathcal{L}$ is weakly functorially dominated by a
uniformly movable category $\mathcal{K}$ then $\mathcal{L}$ is also
uniformly movable.
\end{proposition}

\begin{proof}
For an object $X\in \mathcal{L}$, we can take $M(X)=D(M(J(X))$, and $m_X=\psi (X)\circ D(m_{J(X)}):M(X)\rightarrow X$. Then if $p:Y\rightarrow X$ is an arbitrary morphism in $\mathcal{L}$, we put $u(p)=\psi (Y)\circ D(u(J(p))):M(X) \rightarrow Y$. Now we can verify
the conditions from Definition \ref{umov}. For the condition 1 we have:
\begin{multline*}
p\circ u(p)=[p\circ \psi (Y)] \circ D(u(J(p))) = [\psi (X)\circ D(J(p))] \circ D(u(J(p))) = \\
\psi (X)\circ [J(p)\circ D(u(J(p)))]=\psi (X)\circ D(m_{J(X)})=m_X.
\end{multline*}

Then, for condition 2, if $p:Y\rightarrow X$, $q:Z\rightarrow X$ and $r:Z\rightarrow Y$ are morphisms in $\mathcal{L}$ satisfying $p\circ r=q$, then $J(p)\circ J(r)=J(q)$. This implies $J(r)\circ u(J(q))=u(J(p))$ and by applying the functor $D$ we deduce $D(J(r))\circ D(u(J(q)))=D(u(J(p)))\Rightarrow \psi(Y)\circ D(J(r))\circ D(u(J(q)))= \psi (Y)\circ D(u(J(p))\Rightarrow r\circ \psi (Z)\circ D(u(J(q)))=u(p)$ which means $r\circ u(q)=u(p)$.
\end{proof}

Particularly, Proposition \ref{pr1} holds if $\mathcal{L}$ is functorial dominated by $\mathcal{K}$, i.\,e., $D\circ J=1_{\mathcal{%
L}}$:

\begin{corollary}\label{Cor1}
If the category $\mathcal{L}$ is functorial dominated by a
uniformly movable category $\mathcal{K}$ then $\mathcal{L}$ is also
uniformly movable.
\end{corollary}

\begin{proposition}
A product of categories $\mathcal{K}=\prod\limits_{i\in I}\mathcal{K}_i$ is uniformly movable if and only if every category $\mathcal{K}_i, i\in I$, is uniformly movable.
\end{proposition}

\begin{proof}
Let $\mathcal{K}=\prod\limits_{i\in I}\mathcal{K}_i$ be uniformly movable category. For a fixed index $i_{0}\in I$ and any $i\in I, i\neq i_{0}$, select an object $X_{i}^{0}\in \mathcal{K}_{i}$. Then we consider the following functors: $P_{i_{0}}:$ $\mathcal{K\rightarrow K}_{i_{0}}$, the projection, and $J_{i_{o}}:\mathcal{K}_{i_{0}}\rightarrow \mathcal{K}$ defined by $J_{i_{0}}(X_{i_{0}})=(X_{i}^{\prime })_{i\in I}$, where $X_{i_{0}}^{\prime}=X_{i_{0}}$ and $X_{i}^{\prime }=X_{i}^{0}$, $i\neq i_0$, and for a morphism $f:X_{i_{0}}\rightarrow Y_{i_{0}}$ in $\mathcal{K}_{i_{0}}$, $J_{i_{0}}(f)=(f_{i}^{\prime })_{i\in I}:J_{i_{0}}(X_{i_{0}})\rightarrow J_{i_{0}}(Y_{i_{0}})$ is given by $f_{i}^{\prime }=1_{X_{i}^{0}}$, if $i\neq i_{0}$ and $f_{i_{0}}^{\prime }=f$. Then $P_{i_0}\circ J_{i_0}=1_{\mathcal{K}_{i_0}}$ and we can apply Corollary \ref{Cor1}.

Now, let all categories $\mathcal{K}_i, i\in I,$ are uniformly movable and let us prove that $\mathcal{K}=\prod\limits_{i\in I}\mathcal{K}_i$ is also uniformly movable.

If $X=(X_{i})_{i\in I}$, define $M(X)=(M(X_{i}))_{i\in I}$, with $m_{X}=(m_{X_{i}})_{i\in I}$, and if $p=(p_{i})_{i\in I}:(X_{i})_{i\in I}\rightarrow (Y_{i})_{i\in I}$, then we put $u(p)=(u(p_{i}))_{i\in I}$. With these notations the conditions of Definition \ref{umov} are immediately verified.
\end{proof}

Now let us consider  a category $\mathcal{K}$ with pull-back diagrams. It means that for any pair of morphisms $f:X\to Z$ and $g:Y\to Z$ there exist a commutative diagram
$$\begin{array}{lll}
X\times _{Z}Y & \stackrel{p_{X}}{\rightarrow } & X \\
\downarrow_{p_{Y}} &  & _{f}\downarrow \\
Y & \stackrel{g}{\rightarrow } & Z
\end{array}
$$
called a pull-back diagram, such that for every morphisms $u_{X}:U\rightarrow X$, $u_{Y}:U\rightarrow Y$, satisfying the equality $f\circ u_{X}=g\circ u_{Y}$ there is a unique morphism, which we denote by $u_{X}\times_{Z}u_{Y}:U\rightarrow X\times _{Z}Y$, such that the relations $p_{X}\circ (u_{X}\times _{Z}u_{Y})=u_{X}$ and $p_{Y}\circ
(u_{X}\times _{Z}u_{Y})=u_{Y}$ hold.

\begin{proposition}\label{pr2}
Let $\mathcal{K}$ be a category with pull-back diagrams. If an object $Z\in \mathcal{K}$ is uniformly movable then for any pair of morphisms $f:X\rightarrow Z$ and $g:Y\rightarrow Z$ the following relations hold

\begin{center}
(*) $u(f)\times _{Z}u(g)=u(f\circ p_{X})=u(g\circ p_{Y}).$
\end{center}
\end{proposition}

\begin{proof}
Let $f:X\rightarrow Z$, $g:Y\rightarrow Z$ be two arbitrary morphisms. Consider the morphisms $u(f):M(Z)\rightarrow X$, with $f\circ u(f)=m_{Z}$ and $u(g):M(Z)\rightarrow Y$, with $g\circ u(g)=m_{Z}$ (Diagram 3).

Then the equality $f\circ
u(f)=g\circ u(g)$ permits to consider the morphism $u(f)\times_{Z}u(g):M(Z)\rightarrow X\times _{Z}Y$. Now if we denote $t=f\circ p_{X}=g\circ p_{Y}$, we obtain another morphism $u(t):M(Z)\rightarrow X\times _{Z}Y$. But by condition 2 of Definition \ref{umov}, the relations $t=f\circ p_{X}=g\circ p_{Y}$ imply $u(f)=p_{X}\circ u(t)$ and $u(g)=p_{Y}\circ u.$ These relations and the uniqueness of the morphism $u(f)\times _{Z}u(g)$, prove that $u(t)=u(f)\times _{Z}u(g)$, i.\,e., (*).
\end{proof}

\begin{center}
\begin{picture}(70,100)
\put(0,0){$Y$} \put(80,0){$Z$} \put(0,45){$X\times_{Z}Y$} \put(80,45){$X$} \put(-64,85){$M(Z)$}

\put(10,4){\vector(1,0){60}} \put(35,7){\scriptsize $g$}
\put(47,48){\vector(1,0){25}} \put(50,51){\scriptsize $p_X$}
\put(3,40){\vector(0,-1){25}} \put(6,25){\scriptsize $p_Y$}
\put(85,40){\vector(0,-1){25}} \put(77,25){\scriptsize $f$}
\put(-48,75){\vector(2,-3){40}} \put(-44,35){\scriptsize $u(g)$}
\put(-30,84){\vector(4,-1){100}} \put(10,75){\scriptsize $u(f)$}
\put(-36,77){\vector(3,-2){30}}
\end{picture}

Diagram 3.
\end{center}

\begin{remark}
The dual notions of movability and uniform movability can be defined. An object $X$ of the category $\mathcal{K}$ is (uniformly) co-movable, if there are an object $M(X)\in \mathcal{K}$ and a morphism $m_{X}^{0}:M(X)\leftarrow X$ in $\mathcal{K}$ that satisfy the following condition(s): for any object $Y\in \mathcal{K}$ and any morphism $p:X\rightarrow Y$ in $\mathcal{K}$ there exists a morphism $u^{0}(p):M(X)\leftarrow Y$ in $\mathcal{K}$ satisfying the equality $m_{X}^{0}=u^{0}(p)\circ p$ (and if $p:X\rightarrow Y,q:X\rightarrow Z$ and $r:Y\rightarrow Z$ are morphisms in $\mathcal{K}$ such that $q=r\circ p,$ then $u^{0}(p)=u^{0}(q)\circ r).$ A category $\mathcal{K}$ is called (uniformly) co-movable if all its objects are (uniformly) co-movable. This is equivalent with the fact that the dual category $\mathcal{K}$ is (uniformly) movable.
\end{remark}

\section{Main result}

Recall that if $\mathcal{T}$ is a category, with $\mathcal{P}$ a
subcategory of $\mathcal{T}$ and $X\in \mathcal{T}$, then the
\textsl{comma category} \textsl{of} $X$ \textsl{over} $\mathcal{P}$
is the category denoted by $X_{\mathcal{P}}$ having as objects all
morphisms $p:X\rightarrow P$ in $\mathcal{T}$, with $P\in
\mathcal{P},$ and as morphisms $(X\stackrel{p}{\rightarrow
}P)\rightarrow (X\stackrel{p^{\prime }}{\rightarrow }P^{\prime})$,
all morphisms $u:P\rightarrow P^{\prime }$ in $\mathcal{P}$ such
that the following diagram commutes:

\begin{center}
\begin{picture}(70,90)
\put(0,35){$X$} \put(70,0){$P'.$} \put(70,70){$P$}

\put(13,42){\vector(2,1){48}} \put(30,56){\scriptsize $p$}
\put(13,35){\vector(2,-1){51}} \put(27,15){\scriptsize $p\,'$}
\put(74,64){\vector(0,-1){48}} \put(76,34){\scriptsize $u$}
\end{picture}
\end{center}

\quad
Now recall from \cite{MS2} (Ch. II, \S\, 6,7) the notions of movability and
uniformly movability in terms of inverse systems.

Let $\mathcal{T}$ be a category. Then an object
$\mathbf{X}=(X_{\lambda},p_{\lambda \lambda ^{\prime }},\Lambda )$
of $pro-\mathcal{T}$ is said to be movable provide every $\lambda
\in \Lambda $ admits a $m(\lambda )\geq\lambda $ (called a
\textsl{movability index} of $\lambda $ ) such that any $\lambda
^{\prime \prime }\geq \lambda $ admits a morphism
$r^{\lambda}:X_{m(\lambda )}\rightarrow X_{\lambda ^{\prime \prime
}}$ of $\mathcal{K}$ which satisfies
$$p_{\lambda \lambda ^{\prime\prime }} \circ r^{\lambda}=p_{\lambda,m(\lambda )},$$
i. e., makes the following diagram commutative
\begin{center}
\begin{picture}(70,90)
\put(0,35){$X_\lambda$} \put(70,0){$X_{\lambda''}$} \put(70,70){$X_{m(\lambda)}$}

\put(63,66){\vector(-2,-1){44}} \put(17,60){\scriptsize
$p_{\lambda,m(\lambda)}$} \put(63,10){\vector(-2,1){44}}
\put(27,15){\scriptsize $p_{\lambda\lambda''}$}
\put(74,64){\vector(0,-1){48}} \put(78,34){\scriptsize $r^\lambda$}
\end{picture}
\end{center}

\quad

The essential feature of this condition is that $p_{\lambda
,m(\lambda )}$ factors through $X_{\lambda ^{\prime \prime }}$ for
$\lambda ^{\prime \prime}$ arbitrary large (note that $r^{\lambda }$
is not a bonding morphism).

Then an object $\mathbf{X}=(X_{\lambda },p_{\lambda \lambda
^{\prime}},\Lambda )$ of $pro-\mathcal{K}$ is \textsl{uniformly
movable} if every $\lambda \in \Lambda $ admits a $m(\lambda )\geq
\lambda $ (called a \textsl{uniform} \textsl{movability index} of
$\lambda $ ) such that there is a morphism $\mathbf{r}(\lambda
):\mathbf{X}_{m(\lambda )}\rightarrow \mathbf{X}$ in
$pro-\mathcal{T}$ satisfying
$$\mathbf{p}_{\lambda }\circ \mathbf{r}(\lambda )=p_{\lambda,m(\lambda )},$$
where $\mathbf{p}_{\lambda }:\mathbf{X\rightarrow }X_{\lambda }$ is
the morphism of $pro-\mathcal{T}$ given by $1_{X_{\lambda }}$,
i.\,e., $\mathbf{p}_{\lambda }$ is the restriction of $\mathbf{X}$
to $X_{\lambda }$. Consequently, $p_{\lambda,m(\lambda )}$ factors
through $\mathbf{X}$. Note that the morphism $\mathbf{r}(\lambda)$ determines for every
$\nu \in \Lambda $ a morphism
\begin{center}
$\mathbf{r}(\lambda )^{\nu }:X_{m(\lambda )}\rightarrow X_{\nu },$
\end{center}
in $\mathcal{T},$ such that
\begin{center}
$p_{\nu \nu ^{\prime }}\circ \mathbf{r}(\lambda )^{\nu ^{\prime }}=\mathbf{r}(\lambda )^{\nu },$ if $\nu \leq \nu ^{\prime }$, and $\mathbf{r}(\lambda)^{\lambda }=p_{\lambda,m(\lambda )}.$
\end{center}
In particular, for any $\nu \geq \lambda $ one obtains
$p_{\lambda,m(\lambda )}=\mathbf{r}(\lambda )^{\lambda }=p_{\lambda
\nu }\circ \mathbf{r}(\lambda )^{\nu }$, so that uniform movability
implies movability.

We mention also that the movability and uniform movability for
inverse systems remain valid under isomorphisms of such systems \cite[p. 159--161]{MS2}.

Another notion which we need in this paragraph is that of expansion system of an object.

If $\mathcal{T}$ is a category and $\mathcal{P}$ is a subcategory of $\mathcal{T}$, then for an object $X$ of $\mathcal{T}$, a $\mathcal{P}$-\textsl{expansion} of $X$ is a morphism in $pro-\mathcal{T}$ of $X$ (as rudimentary system) to an inverse system $\mathbf{X}$ $=(X_{\lambda},p_{\lambda \lambda ^{\prime }},\Lambda )$ in $\mathcal{P}$, $\mathbf{p}:X\rightarrow \mathbf{X},$ with the following universal property:

For any inverse system $\mathbf{Y}=(Y_{\mu },q_{\mu \mu ^{\prime }},M)$ in the subcategory $\mathcal{P}$ (called a $\mathcal{P-}$ \textsl{system}) and any morphism $\mathbf{h}:X\rightarrow \mathbf{Y}$ in $pro-\mathcal{T}$, there exists a unique morphism $\mathbf{f:X\rightarrow Y}$ in $pro-\mathcal{T}$ such that $\mathbf{h=f\circ p}$, i.\,e., the following diagram commutes.

\begin{center}
$
\begin{array}{lll}
X & \stackrel{\mathbf{p}}{\rightarrow } & \mathbf{X} \\
& \searrow_{{\mathbf{h}}} & \downarrow _{\mathbf{f}} \\
&  & \mathbf{Y}
\end{array}
$
\end{center}

If $\mathbf{p}:X\rightarrow \mathbf{X}$, $\mathbf{p}':X\rightarrow \mathbf{X}^{\prime}$ are two $\mathcal{P-}$ expansions of the same object $X$, then there is a unique isomorphism $\mathbf{i}:\mathbf{X\rightarrow X}^{\prime}$ such that $\mathbf{i}\circ \mathbf{p}=\mathbf{p}'$. This isomorphism is called the \textsl{natural isomorphism}.

The subcategory $\mathcal{P}$ is called a \textsl{dense} subcategory of the category $\mathcal{T}$ provided every object $X\in \mathcal{T}$ admits a $\mathcal{P}$-expansion $\mathbf{p}:X\rightarrow \mathbf{X}$.

If $\mathbf{p}:X\rightarrow \mathbf{X}$, $\mathbf{p}':X\rightarrow \mathbf{X}^{\prime}$ and $\mathbf{q}:Y\rightarrow \mathbf{Y}$, $\mathbf{q}':Y\rightarrow \mathbf{Y}^{\prime}$ are $\mathcal{P}$-expansions, then two morphisms $\mathbf{f:X\rightarrow Y}$, $\mathbf{f}^{\prime }\mathbf{:X}^{\prime }\mathbf{\rightarrow Y}^{\prime }$, in $pro-\mathcal{P}$, are equivalent, $\mathbf{f}\sim \mathbf{f}^{\prime }$ provided the following diagram in $pro-%
\mathcal{P}$ commutes

\begin{center}
$
\begin{array}{lll}
\mathbf{X} & \stackrel{\mathbf{i}}{\rightarrow } & \mathbf{X}^{\prime } \\
\downarrow _{\mathbf{f}} &  & \downarrow _{\mathbf{f}^{\prime }} \\
\mathbf{Y} & \stackrel{\mathbf{j}}{\rightarrow } & \mathbf{Y}^{\prime }.
\end{array}
$
\end{center}

Now if $\mathcal{P}$ is a dense subcategory of the category $\mathcal{P}$, then the \textsl{shape category} for $(\mathcal{T},\mathcal{P)}$, denoted by $Sh_{(\mathcal{T},\mathcal{P)}},$ has as objects all the objects of $\mathcal{T}$ and morphisms $F:X\rightarrow Y$ are equivalence classes with respect to $\sim $ of morphisms $\mathbf{f:X\rightarrow Y}$ in $pro-\mathcal{P}$, for some $\mathcal{P}$ - expansions $\mathbf{p}:X\rightarrow \mathbf{X}$ and $\mathbf{q}:Y\rightarrow \mathbf{Y}$.

Now, an object $X\in \mathcal{T}$, is said \textsl{movable
(uniformly movable)} \textsl{in }$Sh_{(\mathcal{T},\mathcal{P)}}$ or
simply \textsl{movable (uniformly movable)} if it has a movable
(uniformly movable) $\mathcal{P}$-expansion. This definition is
correct since the properties of movability and uniform movability
for inverse systems are invariant with respect to isomorphisms in
$pro-\mathcal{P}$.

If \textbf{HTop }is the homotopy category of topological spaces, then the homotopy category \textbf{HPol} of polyhedra is a dense subcategory of \textbf{HTop }and a topological space $X$ is called (\textsl{uniformly) movable} if $X$ is \textbf{HPol}-(uniformly) movable.

Now we can establish the main theorem.

\begin{theorem}\label{th1}
Let $\mathcal{T}$ be a category, $\mathcal{P}$ a subcategory of \
$\mathcal{T}$, and let $X\in \mathcal{T}$ be any object and
$\mathbf{p}=(p_{\lambda }):X\rightarrow \mathbf{X}=(X_{\lambda
},p_{\lambda \lambda^{\prime }},\Lambda )$ a $\mathcal{P}$-expansion
of $X$. Then $\mathbf{X}$ is a uniformly movable inverse system if
and only if the comma category $X_{\mathcal{P}}$ of $X$ in
$\mathcal{T}$ over $\mathcal{P}$ is a uniformly movable category.
\end{theorem}

\begin{proof}
Suppose that $X_{\mathcal{P}}$ is a uniformly movable category. We
shall prove that $\mathbf{X}$ is a uniformly movable system.

If $\lambda \in \Lambda $, consider $p_{\lambda }:X\rightarrow X_{\lambda }$ as an object of $X_{\mathcal{P}}$ and there are an object $M(p_{\lambda})=f^{\prime }:X\rightarrow Q^{\prime }$ in $X_{\mathcal{P}}$ and a morphism $m_{p_{\lambda }}=\eta :Q^{\prime }\rightarrow X_{\lambda }$ in $X_{\mathcal{P}}$ satisfying conditions of the Definition \ref{umov} (see below Diagram 4).

By property (AE1) of a $\mathcal{P}$-expansion \cite{MS2} (Ch. I, \S\, 2.1, Th. 1), for $f^{\prime }$ there is a $\widetilde{\lambda }\in \Lambda$, $\widetilde{\lambda }\geq \lambda$ and an $\widetilde{f^{\prime }}:X_{\widetilde{\lambda }}\rightarrow Q^{\prime }$ such that
\begin{equation}\label{1}
f^{\prime }=\widetilde{f^{\prime }}\circ p_{\widetilde{\lambda }}.
\end{equation}

It is not difficult to verify that
\begin{equation}\label{2}
p_{\lambda \widetilde{\lambda }}\circ p_{\widetilde{\lambda }}=\eta \circ \widetilde{f^{\prime }}\circ p_{\widetilde{\lambda }}.
\end{equation}
Indeed:
$$\eta \circ \widetilde{f^{\prime }}\circ p_{\widetilde{\lambda }}=\eta \circ f^{\prime }=p_{\lambda }=p_{\lambda\widetilde{\lambda }}\circ p_{\widetilde{\lambda }}.$$

\begin{center}
\begin{picture}(90,105)
\put(0,30){$X_\lambda$} \put(0,60){$X_{\tilde{\lambda}}$}
\put(45,30){$X$} \put(45,60){$X_{\lambda '}$} \put(90,0){$X_{\lambda
''}$} \put(90,90){$Q'$} \put(87,6){\vector(-3,1){70}}
\put(54,30){\vector(3,-2){35}} \put(41,34){\vector(-1,0){25}}
\put(55,38){\vector(2,3){32}} \put(50,42){\vector(0,1){14}}
\put(5,56){\vector(0,-1){14}} \put(41,63){\vector(-1,0){25}}
\put(87,88){\vector(-3,-2){30}} \put(12,65){\vector(3,1){75}}
\put(94,84){\vector(0,-1){73}}

\put(-10,48){\scriptsize $p_{\lambda \tilde{\lambda}}$}
\put(38,45){\scriptsize $p_{\lambda '}$} \put(95,45){\scriptsize
$u(p_{\lambda \lambda ''})$} \put(20,57){\scriptsize
$p_{\tilde{\lambda} \lambda '}$}
\put(72,57){\scriptsize $f'$}
\put(22,37){\scriptsize $p_\lambda$} \put(43,10){\scriptsize
$p_{\lambda \lambda ''}$} \put(45,80){\scriptsize $\tilde{f'}$}
\put(60,76){\scriptsize $\eta $} \put(70,22){\scriptsize $p_{\lambda
''}$}
\end{picture} { } \\[15pt]

Diagram 4.
\end{center}

From the equality (\ref{2}) and the property (AE2) \cite{MS2} (Ch. I, \S\, 2.1, Th.~ 1), we deduce the existence of an index $\lambda ^{\prime }\in \Lambda$, $\lambda^{\prime }\geq \widetilde{\lambda }$, for which
\begin{equation}\label{3}
p_{\lambda \widetilde{\lambda }}\circ p_{\widetilde{\lambda }\lambda ^{\prime }}=\eta \circ \widetilde{f^{\prime }}\circ p_{\widetilde{\lambda }\lambda ^{\prime }}.
\end{equation}

Now we shall show that the obtained index $\lambda ^{\prime }\in \Lambda$ is a uniform movability index of $\lambda$, i.\,e., we must define a morphism $\mathbf{r}=(r^{\lambda''}):X_{\lambda'}\rightarrow \mathbf{X}$ in $pro-\mathcal{T}$, with $r^{\lambda''}:X_{\lambda^{\prime }}\rightarrow X_{\lambda ^{\prime \prime }}$, $\lambda'' \in \Lambda$, satisfying the condition
\begin{equation}\label{4}
\mathbf{p}_{\lambda }\circ
\mathbf{r}=p_{\lambda \lambda ^{\prime }}.
\end{equation}

Let $\lambda^{\prime \prime }\in \Lambda $ be arbitrary, with $\lambda ^{\prime \prime }\geq \lambda$. For the object $p_{\lambda^{\prime \prime }}:X\rightarrow X_{\lambda ^{\prime \prime }}$ and the morphism $p_{\lambda \lambda ^{\prime \prime }}:X_{\lambda ^{\prime \prime}}\rightarrow X_{\lambda }$ of the comma category $X_{\mathcal{P}}$, there exists a morphism $u(p_{\lambda \lambda ^{\prime \prime }}):Q^{\prime}\rightarrow X_{\lambda ^{\prime \prime }},$ which satisfies the equality
\begin{equation}\label{5}
\eta =p_{\lambda \lambda ^{\prime
\prime }}\circ u(p_{\lambda \lambda ^{\prime \prime }})
\end{equation}
(see Definition 2). Observe, that if $\lambda ^{\prime \prime }=\lambda$ we get

\begin{equation}\label{6}
u(p_{\lambda \lambda })=\eta.
\end{equation}

Now let us define
\begin{equation}\label{7}
r^{\lambda ^{\prime \prime}}=u(p_{\lambda \lambda^{\prime \prime }})\circ \widetilde{f^{\prime }}\circ p_{\widetilde{\lambda }\lambda ^{\prime }}:X_{\lambda ^{\prime }}\rightarrow X_{\lambda ^{\prime\prime }}.
\end{equation}
It is easy to see that
\begin{equation}\label{8}
r^{\lambda}=p_{\lambda \lambda^{\prime }}.
\end{equation}
Indeed, by (\ref{7}), (\ref{6}) and (\ref{3}), we have $r^{\lambda }=u(p_{\lambda \lambda})\circ \widetilde{f^{\prime }}\circ p_{\widetilde{\lambda }\lambda ^{\prime}}=\eta \circ \widetilde{f^{\prime }}\circ p_{\widetilde{\lambda }\lambda^{\prime }}=p_{\lambda \widetilde{\lambda }}\circ p_{\widetilde{\lambda }\lambda ^{\prime }}=p_{\lambda \lambda ^{\prime }}$.

By applying (\ref{3}), (\ref{5}) and (\ref{7}), we get
$$p_{\lambda \lambda ^{\prime }}=p_{\lambda \widetilde{\lambda }}\circ p_{\widetilde{\lambda }\lambda ^{\prime }}=\eta \circ \widetilde{f^{\prime }}\circ p_{\widetilde{\lambda }\lambda ^{\prime }}=p_{\lambda \lambda ^{\prime\prime }}\circ u(p_{\lambda \lambda ^{\prime \prime }})\circ \widetilde{f^{\prime }}\circ p_{\widetilde{\lambda }\lambda ^{\prime }}=p_{\lambda\lambda ^{\prime \prime }}\circ r^{\lambda ^{\prime \prime }}.$$

Thus, for any $\lambda ^{\prime \prime}\in \Lambda$, $\lambda ^{\prime \prime}\geq \lambda$, the following condition is satisfied:
\begin{equation}\label{9}
p_{\lambda \lambda ^{\prime }}=p_{\lambda \lambda^{\prime \prime }}\circ r^{\lambda^{\prime \prime}}.
\end{equation}

Now, for an arbitrary $\lambda^{\prime \prime }\in \Lambda$, with $\lambda ^{\prime \prime }<\lambda <\lambda ^{\prime }$ let us define $r^{\lambda ^{\prime \prime }}=p_{\lambda''\lambda'}$.

In order to show that $\mathbf{r}=(r^{\lambda ^{\prime \prime
}}):X_{\lambda^{\prime }}\rightarrow \mathbf{X}$ is a morphism in
$pro-\mathcal{T}$, which satisfies (\ref{4}), it is necessary to
prove that for any $\lambda ^{\prime\prime }<\lambda ^{\prime \prime
\prime }$ the following condition is satisfied:
\begin{equation}\label{10}
p_{\lambda ^{\prime \prime }\lambda^{\prime \prime \prime }}\circ r^{\lambda'''}=r^{\lambda''}.
\end{equation}

Let $\lambda \leq \lambda ^{\prime \prime }<\lambda^{\prime \prime
\prime }.$ In accordance with (\ref{7}), we have
\begin{equation}\label{11}
r^{\lambda ^{\prime \prime \prime}}=u(p_{\lambda \lambda ^{\prime \prime \prime }})\circ \widetilde{f^{\prime}}\circ p_{\widetilde{\lambda }\lambda ^{\prime }}:X_{\lambda ^{\prime}}\rightarrow X_{\lambda ^{\prime \prime \prime }}.
\end{equation}
Since $p_{\lambda \lambda ^{\prime \prime \prime }}=p_{\lambda \lambda
^{\prime \prime }}\circ p_{\lambda ^{\prime \prime }\lambda ^{\prime \prime\prime }}$, by condition 2 from Definition \ref{umov}, we have
\begin{equation}\label{12}
u(p_{\lambda \lambda ^{\prime\prime }})=p_{\lambda ^{\prime \prime }\lambda ^{\prime \prime \prime}}\circ u(p_{\lambda \lambda ^{\prime \prime \prime }}).
\end{equation}
By applying (\ref{11}), (\ref{12}), (\ref{7}), we get
$$p_{\lambda ^{\prime \prime }\lambda ^{\prime \prime \prime }}\circ r^{\lambda ^{\prime \prime \prime }}=p_{\lambda ^{\prime \prime }\lambda^{\prime \prime \prime }}\circ u(p_{\lambda \lambda ^{\prime \prime \prime}})\circ \widetilde{f^{\prime }}\circ p_{\widetilde{\lambda }\lambda^{\prime }}=u(p_{\lambda \lambda ^{\prime \prime }})\circ \widetilde{f^{\prime }}\circ p_{\widetilde{\lambda }\lambda ^{\prime }}=r^{\lambda^{\prime \prime }}.$$

So, $\mathbf{r}=(r^{\lambda ^{\prime \prime }}):X_{\lambda
^{\prime}}\rightarrow \mathbf{X}$ is a morphism in
$pro-\mathcal{T}$, which satisfies (\ref{4}) and thus $X$ is a
uniformly movable in the sense of shape theory.

Now we shall prove the converse.

Suppose $\mathbf{X}=(X_{\lambda },p_{\lambda \lambda ^{\prime }},\Lambda )$ is a uniformly movable inverse system. We have to verify that $X_{\mathcal{P%
}}$ is a uniformly movable category.

Consider an object $f:X\rightarrow Q$ of the comma category $X_{\mathcal{P}}$ (see Diagram 5). By condition (AE1) it follows that there exist an index $\lambda \in \Lambda $ and an $f_{\lambda }:X_{\lambda }\rightarrow Q$ in $\mathcal{P}$ such that
\begin{equation}\label{13}
f=f_{\lambda }\circ
p_{\lambda }.
\end{equation}

For the index $\lambda \in \Lambda $ let us consider a corresponding uniform movability index $\lambda ^{\prime }\in \Lambda$, $\lambda ^{\prime }\geq\lambda$. From (\ref{13}) we get
\begin{equation}\label{14}
f=f_{\lambda }\circ
p_{\lambda \lambda ^{\prime }}\circ p_{\lambda ^{\prime }}.
\end{equation}

\begin{center}
\begin{picture}(100,120)
\put(50,50){$X$}
\put(0,50){$Q$}
\put(100,50){$X_{\lambda '''}$}
\put(25,0){$Q ''$}
\put(25,100){$X_{\lambda }$}
\put(75,0){$X_{\lambda ''}$}
\put(75,100){$X_{\lambda '}$}

\put(43,53){\vector(-1,0){30}} \put(64,53){\vector(1,0){30}}
\put(50,62){\vector(-1,2){16}} \put(60,44){\vector(1,-2){16}}
\put(50,44){\vector(-1,-2){16}} \put(62,62){\vector(1,2){16}}
\put(24,93){\vector(-1,-2){16}} \put(105,44){\vector(-1,-2){16}}
\put(89,93){\vector(1,-2){16}} \put(24,9){\vector(-1,2){16}}
\put(70,103){\vector(-1,0){30}} \put(70,3){\vector(-1,0){30}}

\put(2,75){\scriptsize $f_\lambda $} \put(32,75){\scriptsize
$p_\lambda $} \put(72,75){\scriptsize $p_{\lambda '}$}
\put(103,75){\scriptsize $r^{\lambda '''}$}

\put(2,25){\scriptsize $\eta '$} \put(32,25){\scriptsize $f ''$}
\put(73,25){\scriptsize $p_{\lambda ''}$} \put(100,25){\scriptsize
$p_{\lambda '' \lambda '''}$}

\put(25,57){\scriptsize $f$} \put(75,57){\scriptsize $p_{\lambda
'''}$} \put(50,107){\scriptsize $p_{\lambda \lambda '}$}
\put(50,-6){\scriptsize $f_{\lambda''}$}
\end{picture} \\[15pt]

Diagram 5.
\end{center}

Now let us prove that the object
$$M(f):=p_{\lambda ^{\prime}}:X\rightarrow X_{\lambda ^{\prime }}$$
and the morphism
\begin{equation}\label{15}
m_{f}:=f_{\lambda }\circ
p_{\lambda \lambda ^{\prime }}:X_{\lambda ^{\prime }}\rightarrow Q
\end{equation}
satisfy the definition of the uniform movability of the comma category $X_{\mathcal{P}}$. Indeed, let $f^{\prime \prime }:X\rightarrow Q^{\prime \prime }$ be an arbitrary object and $\eta ^{\prime }:Q^{\prime \prime }\rightarrow Q$ an arbitrary morphism in $X_{\mathcal{P}},$ i.\,e.,
\begin{equation}\label{16}
f=\eta ^{\prime }\circ f^{\prime
\prime }.
\end{equation}

For the object $f^{\prime \prime }:X\rightarrow Q^{\prime }$ there exist an index $\lambda ^{\prime \prime }\in \Lambda$, $\lambda ^{\prime \prime }\geq\lambda,$ and an $f_{\lambda ^{\prime \prime }}:X_{\lambda ^{\prime \prime}}\rightarrow Q^{\prime \prime }$, such that
\begin{equation}\label{17}
f^{\prime \prime }=f_{\lambda^{\prime \prime }}\circ p_{\lambda ^{\prime \prime }}.
\end{equation}
It is clear that
$$f_{\lambda }\circ p_{\lambda \lambda^{\prime \prime }}\circ p_{\lambda ^{\prime \prime }}=\eta ^{\prime }\circ f_{\lambda ^{\prime \prime }}\circ p_{\lambda ^{\prime \prime }}.$$

Therefore, according to condition (AE2), we can find an index $\lambda^{\prime \prime \prime }\in \Lambda$, $\lambda ^{\prime \prime \prime }\geq\lambda ^{\prime \prime }$, such that
\begin{equation}\label{18}
f_{\lambda }\circ p_{\lambda\lambda ^{\prime \prime }}\circ p_{\lambda ^{\prime \prime }\lambda ^{\prime\prime \prime }}=\eta ^{\prime }\circ f_{\lambda ^{\prime \prime }}\circ p_{\lambda ^{\prime \prime }\lambda ^{\prime \prime \prime }}.
\end{equation}

By the uniform movability of the inverse system $\mathbf{X}=(X_{\lambda},p_{\lambda \lambda ^{\prime }},\Lambda ),$ there exist a morphism $\mathbf{r}=(r^{\lambda ^{\prime \prime \prime }}):X_{\lambda ^{\prime }}\rightarrow\mathbf{X}$ in $pro-\mathcal{T}$ such that $\mathbf{p}_{\lambda }\circ \mathbf{r}=p_{\lambda \lambda ^{\prime}}$, i.\,e., for $r^{\lambda'''}:X_{\lambda'}\rightarrow X_{\lambda'''}$,
\begin{equation}\label{19}
p_{\lambda \lambda ^{\prime}}=p_{\lambda \lambda ^{\prime \prime \prime }}\circ r^{\lambda ^{\prime\prime \prime}}.
\end{equation}

Let us define
\begin{equation}\label{20}
u(\eta ^{\prime})=f_{\lambda ^{\prime \prime }}\circ p_{\lambda ^{\prime \prime }\lambda^{\prime \prime \prime }}\circ r^{\lambda ^{\prime \prime \prime }}.
\end{equation}

By (\ref{20}), (\ref{18}), (\ref{19}) and (\ref{15}) we get
$$\eta ^{\prime }\circ u(\eta ^{\prime })=\eta ^{\prime }\circ f_{\lambda^{\prime \prime }}\circ p_{\lambda ^{\prime \prime }\lambda ^{\prime \prime\prime }}\circ r^{\lambda ^{\prime \prime \prime }}=f_{\lambda }\circ p_{\lambda \lambda ^{\prime \prime }}\circ p_{\lambda ^{\prime \prime}\lambda ^{\prime \prime \prime }}\circ r^{\lambda ^{\prime \prime \prime}}=f_{\lambda }\circ p_{\lambda \lambda ^{\prime }} = m_{f}.
$$

So, the morphism $u(\eta\,'):X_{\lambda'}\rightarrow Q$ satisfies the condition
\begin{equation}\label{21}
\eta\, ^{\prime }\circ u(\eta\,^{\prime })=m_{f}.
\end{equation}
Thus, the first condition of uniform movability of the comma category $X_{\mathcal{P}}$ is proved.

Now we shall prove the second condition of uniform movability of the comma category $X_{\mathcal{P}}$.

Let $\widetilde{f}'':X\rightarrow \widetilde{Q}''$ be an arbitrary object and $\widetilde{\eta}\,':\widetilde{Q}''\rightarrow Q$, $\varphi~:\widetilde{Q}''\rightarrow~ Q^{\prime \prime}$ be some morphisms of the comma category $X_{\mathcal{P}}$ such that
\begin{equation}\label{22}
\widetilde{\eta}\,'=\eta ^{\prime }\circ \varphi.
\end{equation}
Since $\varphi~:\widetilde{Q}''\rightarrow~ Q^{\prime \prime}$ is a morphism in $X_{\mathcal{P}}$,
\begin{equation}\label{1-1}
f''=\varphi \circ \widetilde{f}''.
\end{equation}
In analogy with the construction of the morphism $u(\eta\, ^{\prime })$ (see (\ref{20})) $u(\widetilde{\eta\, ^{\prime }})$ can be written as
\begin{equation}\label{23}
u(\widetilde{\eta}\,')=f_{\widetilde{\lambda}''}\circ p_{\widetilde{\lambda}''\widetilde{\lambda}'''}\circ r^{\widetilde{\lambda}'''}.
\end{equation}
We have also (see (\ref{17}))
\begin{equation}\label{1-2}
\widetilde{f}''=f_{\widetilde{\lambda}''}\circ p_{\widetilde{\lambda}''}.
\end{equation}
It remains to show that
\begin{equation}\label{24}
u(\eta\,')=\varphi \circ u(\widetilde{\eta}\,').
\end{equation}

Consider any index $\lambda _{0}\in \Lambda$, such that $\lambda
_{0}\geq \lambda ^{\prime \prime \prime }$ and $\lambda _{0}\geq
\widetilde{\lambda}\,'''$ (we have that $(\Lambda,\leq )$ is a
directed preordered set). Taking into account (\ref{1-1}) and (\ref{1-2}) it is not difficult to see that
$$f_{\lambda ''}\circ p_{\lambda ''\lambda'''}\circ p_{\lambda '''\lambda_{0}}\circ p_{\lambda _{0}}=\varphi \circ f_{\widetilde{\lambda} ''}\circ p_{\widetilde{\lambda} ''\widetilde{\lambda}'''}\circ p_{\widetilde{\lambda}'''\lambda _{0}}\circ p_{\lambda _{0}}.$$

Hence, by the property (AE2) of the $\mathcal{P}$-expansion $\mathbf{p}=(p_{\lambda }):X\rightarrow \mathbf{X}=(X_{\lambda },p_{\lambda \lambda^{\prime }},\Lambda ),$ we can find an index $\lambda _{1}\in \Lambda$, $\lambda _{1}\geq \lambda _{0},$ such that
\begin{equation}\label{25}
f_{\lambda ^{\prime \prime }}\circ p_{\lambda ^{\prime \prime }\lambda
^{\prime \prime \prime }}\circ p_{\lambda ^{\prime \prime \prime }\lambda_{0}}\circ p_{\lambda _{0}\lambda _{1}}=\varphi \circ f_{\widetilde{\lambda} ^{\prime \prime }}\circ p_{\widetilde{\lambda} ^{\prime \prime }\widetilde{\lambda} ^{\prime \prime \prime }}\circ p_{\widetilde{\lambda}^{\prime \prime \prime }\lambda _{0}}\circ p_{\lambda _{0}\lambda _{1}}.
\end{equation}
Since $\mathbf{r}=(r^{\lambda ^{\prime \prime \prime }}):X_{\lambda ^{\prime}}\rightarrow \mathbf{X}$ is a morphism of inverse systems, and $\lambda_{1}\geq \lambda ^{\prime \prime \prime }$, $\lambda _{1}\geq \widetilde{\lambda} ^{\prime \prime \prime }$, we have
\begin{equation}\label{26}
r^{\lambda ^{\prime \prime \prime}}=p_{\lambda ^{\prime \prime \prime }\lambda _{0}}\circ p_{\lambda
_{0}\lambda _{1}}\circ r^{\lambda _{1}}
\end{equation}
and
\begin{equation}\label{27}
r^{\widetilde{\lambda} ^{\prime \prime\prime }}=p_{\widetilde{\lambda} ^{\prime \prime \prime }\lambda _{0}}\circ
p_{\lambda _{0}\lambda _{1}}\circ r^{\lambda _{1}}.
\end{equation}

Now we are ready to verify (\ref{24}). By (\ref{20}), (\ref{23}), (\ref{25}), (\ref{26}) and (\ref{27}), we get
\begin{multline*}
\varphi \circ u(\widetilde{\eta}\, ^{\prime })=\varphi \circ f_{\widetilde{\lambda} ^{\prime \prime }}\circ p_{\widetilde{\lambda}\, ^{\prime \prime }\widetilde{\lambda} ^{\prime \prime \prime }}\circ r^{\widetilde{\lambda}\,^{\prime \prime \prime }}=\varphi \circ f_{\widetilde{\lambda} ^{\prime\prime }}\circ p_{\widetilde{\lambda} ^{\prime \prime }\widetilde{\lambda}^{\prime \prime \prime }}\circ p_{\widetilde{\lambda} ^{\prime \prime \prime}\lambda _{0}}\circ p_{\lambda _{0}\lambda _{1}}\circ r^{\lambda _{1}}= \\
f_{\lambda ^{\prime \prime }}\circ p_{\lambda ^{\prime \prime }\lambda^{\prime \prime \prime }}\circ p_{\lambda ^{\prime \prime \prime }\lambda_{0}}\circ p_{\lambda _{0}\lambda _{1}}\circ r^{\lambda _{1}}=f_{\lambda^{\prime \prime }}\circ p_{\lambda ^{\prime \prime }\lambda ^{\prime \prime\prime }}\circ r^{\lambda ^{\prime \prime \prime }}=u(\eta\, ^{\prime }).
\end{multline*}

This completes the proof of the theorem.
\end{proof}

Now by using some theorems of dense subcategory (see \cite{MS2}, Th. 2, Ch.
I, \S\, 4.1, Th. 6 and Th. 7, Ch. I, \S\, 4.3), we can formulate the
following corollaries.

\begin{corollary}\label{cor2}
Let $\mathcal{T}$ be a category and $\mathcal{P}$ a dense
subcategory of \ $\mathcal{T}$. An object $X\in \mathcal{T}$ is
uniformly movable in the sense of shape theory if and only if the
comma category $X_{\mathcal{P}}$ of $X$ in $\mathcal{T}$ over
$\mathcal{P}$ is uniformly movable category.
\end{corollary}

\begin{corollary}
If $\mathcal{P}$ is a dense subcategory of a category $\mathcal{T}$
then any object $P\in \mathcal{P}$ is uniformly movable.
\end{corollary}

\begin{proof}
The comma category $P_\mathcal{P}$ has as initial object the
identity morphism $1_P:P\to P$. By Proposition \ref{pr3}, this
implies that $P_\mathcal{P}$ is uniformly movable category and we
can apply Corollary \ref{cor2}.
\end{proof}

\begin{corollary}
An arbitrary topogical space $X$ is uniformly movable if and only if
its comma category $X_{\mathbf{HPol}}$ in the category \textbf{HTop
}over the subcategory \textbf{HPol} is uniformly movable.

In particular polyhedra and ANR's are uniformly movable spaces.
\end{corollary}

\begin{corollary}
An arbitrary pair of topological spaces $(X,X_{0})$ is uniformly
movable if and only if its comma category
$(X,X_{0})_{\mathbf{HPol}^{2}}$ in the homotopy category of pairs
\textbf{HTop}$^{2}$ over the homotopy subcategory of polyhedral
pairs \textbf{HPol}$^{2}$ is uniformly movable.

Particularly, a pointed space $(X,*)$ is uniformly movable if and
only if its comma category $(X,*)_{\mathbf{HPol}_{*}}$ in the
pointed homotopy category \textbf{HTop}$_{*}$ over the pointed
homotopy subcategory of polyhedra \textbf{HPol}$_{*}$ is uniformly
movable.

All pointed polyhedra and pointed ANR's are uniformly movable.
\end{corollary}

\begin{remark}
In \cite{G1}  a similar theorem with Theorem \ref{th1} was stated for the
movable space. Precisely it was proved that a topological space $X$
is a movable space if
and only if its comma category $X_{\mathbf{HPol}}$ in the category \textbf{%
HTop }over the subcategory \textbf{HPol} is movable. Now we can use
the fact that there are movable objects which are not uniformly
movable (see \cite{MS2}, p. 255) to conclude that there are movable
categories which are not uniformly movable.
\end{remark}

By the main Theorem \ref{th1} and Theorem 4 from [4, p. 173] we have
the following particular case.

\begin{corollary}
Let $\mathcal{T}$ be a category, $\mathcal{P}$ a subcategory of \ $\mathcal{T}$ and $X\in \mathcal{T}$.
Suppose that $X$ has as a $\mathcal{P}$-expansion an inverse sequence $\mathbf{p}:X\rightarrow\mathbf{X}=(X_{n},p_{nn+1})$.
Then the comma category $X_\mathcal{P}$ of $X$  in $\mathcal{T}$ over $\mathcal{P}$  is uniformly movable if and only if it is movable.
\end{corollary}

\begin{remark}
As we have specified in Introduction, if we take into consideration the more
restrictive definition for the uniform movability of a category given in
\cite{P}, in order to prove the uniform movability of the comma category $X_{%
\mathcal{P}}$ of an object $X,$ two supplementary conditions were
added to the uniform movability of a $\mathcal{P}$-expansion
$\mathbf{X=}(X_{\lambda },p_{\lambda \lambda ^{\prime }},\Lambda )$
of $X$. These conditions are the following.

(G1) If $m(\lambda )$ is the uniform movability index of $\lambda,$
then $p_{\lambda,m(\lambda )}:X_{m(\lambda )}\rightarrow X_{\lambda
}$ is a $\mathcal{P}$-monomorphism, that is if $p_{\lambda
,m(\lambda )}\circ u=p_{\lambda,m(\lambda )}\circ v,$ where
$u,v:P\rightarrow X_{m(\lambda )}$ are two morphisms in the
subcategory $\mathcal{P},$ then $u=v.$

(G2) If $\lambda,\lambda ^{\prime }\in \Lambda,$ then there exists a $%
\lambda ^{*}\in \Lambda,$ with $\lambda ^{*}\geq m(\lambda
),m(\lambda ^{\prime }),$ such that the following diagram commutes.

\begin{center}
\begin{picture}(100,120)
\put(50,100){$X_{m(\lambda )}$} \put(0,50){$X_{\lambda ^*}$}
\put(100,50){$\mathbf{X}$} \put(50,0){$X_{m(\lambda')}$}

\put(17,62){\vector(1,1){30}} \put(67,14){\vector(1,1){30}}
\put(17,45){\vector(1,-1){30}} \put(67,93){\vector(1,-1){30}}

\put(0,77){\scriptsize $p_{m(\lambda ),\lambda^*}$}
\put(87,77){\scriptsize $\mathbf{r}_{\lambda}$}
\put(0,24){\scriptsize $p_{m(\lambda'),\lambda^*}$}
\put(87,24){\scriptsize $\mathbf{r}_{\lambda'}$}
\end{picture} \\[15pt]
\end{center}

A space admitting such a $\mathcal{P}$-expansion was called $\mathcal{P}$-\textsl{global uniformly movable}. An example is the Warsaw circle \cite[Ex. 9]{P}.

These specifications motivate the timeliness, given by the first author, of the new definition of the uniform movability of a category which we apply in this paper.\textbf{\ }
\end{remark}

\end{document}